\RequirePackage[l2tabu, orthodox]{nag}
\documentclass[preprint,  11pt]{amsart}

\usepackage{fullpage}
\usepackage{xcolor}
\usepackage{lipsum}
\usepackage{bbm}
\usepackage{tikz-network}
\usepackage[square,numbers]{natbib}
\usepackage{libertinus}

\usepackage[T1]{fontenc}
\usepackage{bm}
\usepackage{pdfpages}
\usepackage{pgfplots}
\pgfplotsset{compat=1.17}
\usetikzlibrary{positioning}

\usepackage{gfsartemisia}
\usepackage[T1]{fontenc}
\usepackage{subfigure}

\usepackage{amsfonts}
\usepackage{amssymb}
\usepackage{amsmath}
\usepackage{amsthm}
\usepackage{amsbsy}
\usepackage{amssymb} 
\usepackage{verbatim}
\usepackage{bm} 
\usepackage{paralist} 
 \usepackage{color}
 \usepackage{mathrsfs}
 \usepackage{graphicx}
\usepackage{hyperref}

\newcommand{\E}{\mathbf{E}}
\def\P{\mathbf{P}}

\def\l{\left}
\def\r{\right}
\def\<{\langle}
\def\>{\rangle}

\newcommand{\ba}{\[\begin{aligned}}
\newcommand{\ea}{\end{aligned}\]}
\newcommand\mnote[1]{} 
\newcommand{\beq}[1]{\begin{equation}\label{#1}}
\newcommand\eeq{\end{equation}}
\newcommand\ben{\begin{equation}}
\newcommand\een{\end{equation}}
\newcommand\bes{\begin{eqnarray*}}
\newcommand\ees{\end{eqnarray*}}
\newcommand\besn{\begin{eqnarray}}
\newcommand\eesn{\end{eqnarray}}

\def\bthm{\begin{theorem}}
\def\ethm{\end{theorem}}
\def\bdefn{\begin{definition}}
\def\edefn{\end{definition}}
\newcommand{\benu}{\begin{enumerate}\setlength\itemsep{6pt}}
\newcommand{\beit}{\begin{itemize}\setlength\itemsep{3pt}}
\def\eenu{\end{enumerate}}
\def\eeit{\end{itemize}}
\def\beds{\begin{description}}
\def\eeds{\end{description}}
\def\bepr{\begin{problem}}
\def\eepr{\end{problem}}

\def\bprf{\begin{proof}}
\def\eprf{\end{proof}}
\def\berk{\begin{remark}}
\def\eerk{\end{remark}}
\def\bex{\begin{exercise}}
\def\eex{\end{exercise}}
\def\beg{\begin{example}}
\def\eeg{\end{example}}

\def\N{\mathbb{N}}
\def\R{\mathbb{R}}

\newcommand{\sm}{{\raise0.3ex\hbox{$\scriptstyle \setminus$}}}

\renewcommand\phi{\varphi}

\theoremstyle{plain} 
    \newtheorem{theorem}{Theorem}
    
    \newtheorem{proposition}{Proposition}
    \newtheorem{corollary}{Corollary}
    
    \newtheorem{conjecture}{Conjecture}

\theoremstyle{definition} 
    \newtheorem{definition}[theorem]{Definition}

    \newtheorem{exercise}[theorem]{Exercise}
    \newtheorem{problem}[theorem]{Problem}
        \newtheorem{remark}{Remark}
    \newtheorem{example}[theorem]{Example}

\renewcommand\P{\mathbb P}
\renewcommand\E{\mathbb E}

\renewcommand{\bex}{\indent\begin{exercise}}

\hypersetup{
    colorlinks=true, 
    linktoc=all,     
    linkcolor=blue,  
}
\begin{document}
\title{Stochastic domination in beta ensembles}
\author{Jnaneshwar Baslingker}

\begin{abstract}
    We give a stochastic comparison and ordering of the largest eigenvalues, with parameter $\beta$, for Hermite $\beta$-ensembles and Laguerre $\beta$-ensembles. Although stochastic comparison results are well known in Laguerre ensembles (for $\beta=1,2,4$) using the last passage percolation models, our results are novel even for $\beta=1,2,4$, in Hermite ensembles. Taking limit, we recover a stochastic domination result for Tracy-Widom distributions obtained in \cite{VP22}. Using this, we also obtain a result on the signs of means of Tracy-Widom distributions. Our methods also provide stochastic domination results for spiked beta ensembles as well. We compare ordering of all the eigenvalues collectively, with $\beta$ as a parameter, by proving ordering of the moments of Hermite and Laguerre $\beta$-ensembles.
    
    In order to generalize the stochastic domination results of \cite{VP22} to higher order analogues of Tracy-Widom distributions, we study tail estimates of these distributions. We show that the description of these distributions as eigenvalues of a stochastic operator is inconsistent with the known tail estimates. As a result, we disprove a conjecture of \cite{KRV13} for $\beta=2$ and $k=1$. 
\end{abstract}

\address{Department of Mathematics, Indian Institute of Science, Bangalore, India}
 \email{jnaneshwarb@iisc.ac.in}
\keywords{ Stochastic domination, Random matrices}
\subjclass[2010]{60B20, 60B05}

\maketitle

\section{Introduction and main results}

The eigenvalues of Gaussian matrices (GUE/GOE) and Wishart matrices (LUE/LOE) have been extensively studied. These eigenvalue ensembles are special cases of a general class of point process called the $\beta$-ensembles (for a real parameter $\beta>0$). In this article we consider Hermite and Laguerre $\beta$-ensembles which are the most well-studied random matrix models over the last few decades \cite{BF97,DE02,ES07,KJ,RRV11}.

For any real $\beta>0$, consider the probability density function of $\lambda_1\geq \lambda_2\geq\dots\geq\lambda_n\in \mathbb{R}$ given by 
\begin{align}\label{Hermite density}
    \mathbb{P}_{n}^{\beta}\left(\lambda_1,\lambda_2,\dots,\lambda_n\right)=\frac{1}{Z_n^{\beta}}\prod_{k=1}^{n}e^{-\frac{\beta}{4}\lambda_k^2}\times\prod\limits_{j<k}\left|\lambda_j-\lambda_k\right| ^{\beta}.
\end{align}

For $\beta=1,2$ or $4$, this distribution corresponds to the joint density of eigenvalues of $n\times n$ Gaussian orthogonal, unitary or symplectic matrices, G(O/U/S)E of (for general $\beta$ this density is also known as Hermite $\beta$-ensemble) random matrix theory. For $\kappa\in\mathbb{R}$ with $\kappa>n-1$, consider the following joint density on points $\lambda_1\geq\lambda_2\geq\dots\geq\lambda_n\geq 0$,
\begin{align}\label{Laguerre density}
    \mathbb{P}_{n,\kappa}^{\beta}\left(\lambda_1,\lambda_2,\dots,\lambda_n\right)=\frac{1}{Z_{n,\kappa}^{\beta}}\prod_{k=1}^{n}\lambda_k^{\frac{\beta}{2}(\kappa-n+1)-1}e^{-\frac{\beta}{2}\lambda_k}\times \prod\limits_{j<k}\left|\lambda_j-\lambda_k\right| ^{\beta}.
\end{align}

When $\kappa$ is an integer and $\beta=1,2$ or $4$, this distribution corresponds to the joint density of eigenvalues of  Laguerre orthogonal, unitary or symplectic matrices, L(O/U/S)E of random matrix theory (for general $\beta$ this density is also known as Laguerre $\beta$-ensemble)\cite{DE02}.

The limiting distribution of the largest eigenvalues of Hermite $\beta$-ensembles \eqref{Hermite density} and Laguerre $\beta$-ensembles \eqref{Laguerre density} are some of the most celebrated results in random matrix theory due to their connections to physics \cite{PS02}, combinatorics \cite{BDJ99}, statistics \cite{J08} and applied probability \cite{KJ, TW09}. Let $\lambda_1^{(n,H,\beta)}$ and $\lambda_1^{(n,\kappa,L,\beta)}$ denote $\lambda_1$ in \eqref{Hermite density} and in \eqref{Laguerre density} respectively. Define
\begin{align*}
    H_{n,\beta}&:=\left(\frac{\lambda_1^{(n,H,\beta)}}{\sqrt{ n}}-2\right)n^{2/3},\\
    L_{n,\kappa,\beta}&:=\l(\sqrt{\kappa n}\r)^{1/3}{\left(\sqrt{\kappa}+\sqrt{n}\right)^{2/3}}\left(\frac{\lambda_1^{(n,\kappa, L,\beta)}}{\left(\sqrt{\kappa}+\sqrt{n}\right)^{2}}-1\right).
\end{align*}

Ram\'{i}rez, Rider and Vir\'{a}g show in \cite{RRV11} that for $\beta>0$, if $n\rightarrow \infty$, then $H_{n,\beta}$ converges weakly to a non-trivial distribution known as Tracy-Widom distribution with parameter $\beta$, denoted as $TW_{\beta}$. They also show that for $\beta>0$, if $n\rightarrow \infty$ with $\kappa=\kappa_n>n-1$, then $L_{n,\kappa,\beta}$ converges in distribution to $TW_{\beta}$.

\subsection{Stochastic domination in beta ensembles:}
Despite the ubiquity of $TW_{\beta}$ distributions, very few properties are known about them.
In a recent work, Virginia Pedreira \cite{VP22} showed that if $s\in\left [ 1/3, 2/3 \right]$ and $\beta_2>\beta_1$ then  $\beta_1^{s}TW_{\beta_1} \ {\succeq} \ \beta_2^{s}TW_{\beta_2}$ which shows the stochastic domination of $TW_{\beta_1}$ over $TW_{\beta_2}$ modulo a multiplicative constant. The above stochastic domination result was obtained using the description of $TW_{\beta}$ as the distribution of the negative of the smallest eigenvalue  of stochastic Airy operator mentioned in \cite{RRV11}. 

As $H_{n,\beta}$ and $L_{n,\kappa,\beta}$ both converge weakly to $TW_{\beta}$, it is a natural question whether there exists a stochastic domination result similar to that of \cite{VP22} for $H_{n,\beta}$ and $L_{n,\kappa,\beta}$. Our first main result is following theorem for Hermite and Laguerre $\beta$-ensembles.
\begin{theorem}\label{Stoch dom}
Let $\frac{\beta_2}{\beta_1}=\frac{m}{n}\in \mathbb{Q}$ and $\beta_2>\beta_1$. Let $s=2/3$. Then
\begin{align}
    \beta_1^s\  H_{m,\beta_1}&\ \succeq\  \beta_2^s\  H_{n,\beta_2}. \label{Hermite} \\
     \beta_1^s\  L_{m,m\kappa/n,\beta_1}&\ \succeq\  \beta_2^s\  L_{n,\kappa,\beta_2}\label{Laguerre}.
\end{align}
\end{theorem}

As a corollary of Theorem \ref{Stoch dom}, we recover the following result of Pedreira \cite{VP22} for $s=2/3$.

\begin{corollary}\label{Pedreira result}
Let $s=2/3$ and $\beta_2>\beta_1$. Then $\beta_1^sTW_{\beta_1}\ \succeq\ \beta_2^sTW_{\beta_2}$.
\end{corollary}

\begin{remark}
Although \eqref{Laguerre} was known for $\beta=1,2,4$, due to the connection to LPP models, the result \eqref{Hermite} is novel even for $\beta=1,2,4$. As Theorem \ref{Stoch dom} is for finite $n$, it strengthens the main result of \cite{VP22} for $s=2/3$. Note that it is enough to prove the main result of \cite{VP22} for $s\in\l\{\frac{1}{3},\frac{2}{3}\r\}$. It would be interesting to know if a similar result would also hold for $s=1/3$ for finite $n$ in either Hermite $\beta$-ensemble or Laguerre $\beta$-ensemble. This would strengthen the main result of \cite{VP22} for $s=1/3$. It would also be interesting to know if such stochastic domination results hold for $j$-th largest eigenvalue $\lambda_j$ (our method works only for $\lambda_1$). This would allow us to obtain stochastic ordering of limiting distributions of $\lambda_j$ obtained in \cite{VP22}.
\end{remark}

Using stochastic domination of $TW_{\beta}$, we also prove the following qualitative result about the means of these distributions. 

\begin{theorem}\label{thm: TW neg means}
    For all $\beta\geq 1$, we have $\mathbb{E}[TW_{\beta}]<0$. Further, there exists some $0<a<1$ such that for all $\beta<a$, the means of $TW_{\beta}$ are non-negative.
\end{theorem}

 We prove Theorem \ref{Stoch dom} using the tridiagonal matrix model for \eqref{Hermite density}, \eqref{Laguerre density}. As the proof of Theorem \ref{Stoch dom} uses only the tridiagonal matrix model, this makes the proof stronger than that of \cite{VP22}, which uses the stochastic Airy operator and the variational characterization of its eigenvalues. Also our method can be applied to more general random tridiagonal matrix models, without knowing the exact description of limiting distributions of fluctuations of largest eigenvalue.

Dumitriu and Edelman \cite{DE02} gave a tridiagonal matrix model for \eqref{Hermite density} and \eqref{Laguerre density}. Define 
\begin{align}\label{eq: tridiag model}
 T_{n,\beta}:=\frac{1}{\sqrt{\beta}}\begin{bmatrix}
X_1 & Y_1 & 0 &\cdots 0\\
Y_1 & X_2 & Y_2 & \cdots 0\\
0 & Y_2 & X_3 &\cdots 0\\
\vdots & \vdots & \vdots &\vdots
\end{bmatrix}_{n\times n}
    \hspace{2cm}
B_{n,\kappa,\beta}:=\frac{1}{\sqrt{\beta}}\begin{bmatrix}
Z_1 & 0 & 0 &\cdots 0\\
W_1 & Z_2 & 0 & \cdots 0\\
0 & W_2 & Z_3 &\cdots 0\\
\vdots & \vdots & \vdots &\vdots
\end{bmatrix}_{n\times n}.
\end{align}
$T_{n,\beta}$ is the symmetric tridiagonal matrix with $X_i\sim N(0,2)$ and $Y_i\sim \chi_{\beta(n-i)}$ and the entries are independent up to symmetry.
 $B_{n,\kappa, \beta}$ is the bidiagonal matrix  with $Z_i\sim \chi_{\beta(\kappa-i+1)}$ and $W_i\sim \chi_{\beta(n-i)}$, where $\kappa\in \mathbb{R}$ and $\kappa+1>n$ and all the entries are independent. It is shown in \cite{DE02} that the $n$ eigenvalues of $T_{n,\beta}$ have the joint law given by \eqref{Hermite density} and the eigenvalues of $B^t_{n,\kappa, \beta}B_{n,\kappa, \beta}$ have the joint density given by \eqref{Laguerre density}. In particular, if $\lambda_1(A)$ denotes the largest eigenvalue of a matrix $A$, 
 \begin{align*}
     \lambda_1(T_{n,\beta})\overset{d}{=}\lambda_1^{(n,H,\beta)}\quad  \quad\quad \lambda_1(B^t_{n,\kappa, \beta}B_{n,\kappa, \beta})\overset{d}{=}\lambda_1^{(n,\kappa,L,\beta)}.
 \end{align*} 
We then prove the stochastic domination results for these largest eigenvalues of tridiagonal matrices. We describe now how our method allows stochastic comparison in spiked random matrix models as well.

\subsection{Stochastic domination in spiked beta ensembles:}
The main result of \cite{VP22} was obtained using the stochastic Airy operator (SAO) and the variational characterization of its eigenvalues. SAO is the random  operator: 
\begin{align*}
    S_{\beta}:=-\frac{d^2}{dx^2}+x+\frac{2}{\sqrt{\beta}}b_x'
\end{align*}
where $b'$ is the white noise. This operator is defined on the Hilbert space of continuous functions $f$ such that $f(0)=0$ and $\int_0^{\infty}(f'(x))^2+(1+x)f^2(x)dx<\infty$ (see \cite{RRV11}). Taking $\Lambda_0$ as the smallest eigenvalue, $TW_{\beta}$ is the distribution of $-\Lambda_0$. Similarly the spiked SAO $S^{(w)}_{\beta}$, which is same as $S_{\beta}$ but with boundary condition $f'(0)=wf(0)$ has been studied in the work of Blomendal and Vir\'ag \cite{BV12}. For finite $w$, the distribution of the negative of the smallest eigenvalue of these operators gives a family of distributions denoted as $TW_{\beta}^{(w)}$, which are deformations of $TW_{\beta}$ and $w=\infty$ corresponds to the Dirichlet boundary condition $f(0)=0$. Carrying out the same analysis of \cite{VP22}, one can obtain stochastic comparison for $TW_{\beta}^{(w)}$ also. 

 Consider the random matrices 
 \begin{align}\label{eq:spiked tridiag model}
 T_{n,\beta}^{(\ell_n)}:=\frac{1}{\sqrt{\beta}}\begin{bmatrix}
X_1+ \sqrt{\beta n}\ell_n & Y_1 & 0 &\cdots 0\\
Y_1 & X_2 & Y_2 & \cdots 0\\
0 & Y_2 & X_3 &\cdots 0\\
\vdots & \vdots & \vdots &\vdots
\end{bmatrix}_{n\times n}
    \hspace{1cm}
B_{n,\kappa,\beta}^{(\ell_{n,\kappa})}:=\frac{1}{\sqrt{\beta}}\begin{bmatrix}
\sqrt{\ell_{n,\kappa}}Z_1 & 0 & 0 &\cdots 0\\
W_1 & Z_2 & 0 & \cdots 0\\
0 & W_2 & Z_3 &\cdots 0\\
\vdots & \vdots & \vdots &\vdots
\end{bmatrix}_{n\times n}
\end{align}
where all the random variables are same as in the definition of $T_{n,\beta}$ and $B_{n,\kappa,\beta}$ in \eqref{eq: tridiag model} and 
\begin{align*}
    \lim_{n\rightarrow\infty}n^{1/3}\l(1-\ell_n\r)&=w\in (-\infty,\infty]\\
\lim_{n\rightarrow\infty}\l(\kappa^{-1/2}+n^{-1/2}\r)^{-2/3}\l(1-\sqrt{\kappa/n}\l(\ell_{n,\kappa}-1\r)\r)&=w\in (-\infty,\infty] \mbox{ with  } \kappa>n-1.
\end{align*}
In \cite{BV12}, it is shown that for all $\beta>0$, as $n\rightarrow\infty$
\begin{align*}
H_{n,\beta}^{(\ell_n)}:=\l(\frac{\lambda_1\l(T_{n,\beta}^{(\ell_n)}\r)}{\sqrt{n}}-2\r)n^{2/3}\rightarrow TW_{\beta}^{(w)}\\
L_{n,\kappa,\beta}^{(\ell_{n,\kappa})}:=\l(\sqrt{\kappa n}\r)^{1/3}{\left(\sqrt{\kappa}+\sqrt{n}\right)^{2/3}}\left(\frac{\lambda_1\l(\l({B_{n,\kappa, \beta}^{(\ell_{n,\kappa})}}\r)^t\l({B_{n,\kappa, \beta}^{(\ell_{n,\kappa})}}\r)\r)}{\left(\sqrt{\kappa}+\sqrt{n}\right)^{2}}-1\right)\rightarrow TW_{\beta}^{(w)}.
\end{align*} 
The proof method of Theorem \ref{Stoch dom} can be followed verbatim to obtain the following result.
\begin{theorem}\label{Stoch dom spiked}
Let $\frac{\beta_2}{\beta_1}=\frac{m}{n}\in \mathbb{Q}$ and $\beta_2>\beta_1$. Let $s=2/3$ and $\ell_n,\ell_{n,\kappa}$ be non-decreasing. Then
\begin{align*}
    \beta_1^s\  H^{\ell_n}_{m,\beta_1}&\ \succeq\  \beta_2^s\  H^{\ell_n}_{n,\beta_2}.  \\
     \beta_1^s\  L^{\ell_{n,\kappa}}_{m,m\kappa/n,\beta_1}&\ \succeq\  \beta_2^s\  L^{\ell_{n,\kappa}}_{n,\kappa,\beta_2}.
\end{align*}
\end{theorem}

\subsection{Motivation and known results:}

The stochastic domination results for Laguerre ensemble ($\beta=1,2,4$), using the connection to last passage percolation (LPP) models, are well known. Due to the work of Johannson \cite{KJ}, Baik and Rains \cite{BR01}, \cite{Baik}, there are remarkable bijections between the passage times of a few last passage percolation (LPP) models and the largest eigenvalues of Laguerre ensembles. Stochastic domination results are immediate for $\beta=1,2,4$ in Laguerre ensemble using these connections, as shown below. But there are no known probabilistic interpretations in LPP models for $\beta\notin\{1,2,4\}$. Also in the case of Hermite ensemble, except for $\beta=2$, there are no known last passage interpretations of the largest eigenvalues.  Theorem \ref{Stoch dom} generalizes these stochastic domination results from last passage percolation to all $\beta>0$ and also introduce new stochastic domination results in the case of Hermite ensembles.

Consider the exponential LPP model on $\mathbb{Z}^2$ where the field of vertex weights $\{\zeta_v\}_{v\in\mathbb{Z}^2}$ is a family of i.i.d. rate one exponentially distributed random variables. We shall define
\begin{align*}
G^{\scalebox{1.5}{$\Box$}}_{u,v}:=\max\limits_{\gamma}\ell(\gamma),
\end{align*}
where $\gamma$ are the up/right oriented paths in $\mathbb{Z}^2$ from $u\in\mathbb{Z}^2$ to $v\in\mathbb{Z}^2$ and $\ell(\gamma)=\sum\limits_{v\in \gamma}\zeta_v$. Let $\mathbf{1}=(1,1)\in \mathbb{Z}^2$, $\mathbf{n}=(n,n)\in \mathbb{Z}^2$. Define 
\begin{align*}
G^{\scalebox{1.5}{$\boxbslash$}}_{\mathbf{1},\mathbf{n}}:=\sup\limits_{\mathbf{u}\in\mathbb{Z}_+^2:u_1+u_2=2n}G^{\scalebox{1.5}{$\Box$}}_{\mathbf{1},\mathbf{u}}.     
\end{align*}
Using the relation $G^{\scalebox{1.5}{$\Box$}}_{\mathbf{1},\mathbf{n}}\overset{d}{=}\lambda_{1}^{(n,n,L,2)}$ \cite{KJ} and $G^{\scalebox{1.5}{$\boxbslash$}}_{\mathbf{1},\mathbf{n}}\overset{d}{=}\frac{1}{2}\lambda_{1}^{(2n-1,2n,L,1)}$ (Proposition $1.3$, \cite{BGHK}) and $G^{\scalebox{1.5}{$\boxbslash$}}_{\mathbf{1},\mathbf{n}}\succeq G^{\scalebox{1.5}{$\Box$}}_{\mathbf{1},\mathbf{n}}$, one can see that 
\begin{align}
\frac{1}{2}\lambda_{1}^{(2n-1,2n,L,1)}\ \succeq \ \lambda_{1}^{(n,n,L,2)}.\label{1,2 stoch}    
\end{align}

This gives a stochastic comparison between the largest eigenvalues in Laguerre ensemble for $\beta=2$ and $\beta=1$.

Consider a last passage model on $\mathbb{Z}^2$, having i.i.d $\zeta_v\sim$ Exp$(1)$ random weights on vertices, with symmetry across the $x=y$ line and zero weights on the $x=y $ line, that is $\zeta(i,j)=\zeta(j,i)$ and $\zeta(i,i)=0$. Define
\begin{align*}
G^{\scalebox{1.5}{$\boxslash$}}_{\mathbf{1},\mathbf{n}}:=\max\limits_{\gamma}\ell(\gamma),
\end{align*}
where $\gamma$ are the up/right oriented paths in $\mathbb{Z}^2$ from $\mathbf{1}\in\mathbb{Z}^2$ to $\mathbf{n}\in\mathbb{Z}^2$ and $\ell(\gamma)=\sum\limits_{v\in \gamma}\zeta_v$.

Using \eqref{Laguerre density}, we also have the following relation due to Baik \cite{Baik}.

\begin{proposition}[{\cite[Corollary 1.5]{Baik}}]\label{Baik Prop}
\begin{align*}
G^{\scalebox{1.5}{$\boxslash$}}_{\mathbf{1},\mathbf{2n}}\overset{d}{=}  2\lambda_{1}^{(n,n-\frac{1}{2},L,4)}
\end{align*}
\end{proposition}

Using Proposition \ref{Baik Prop} and $G^{\scalebox{1.5}{$\Box$}}_{\mathbf{1},\mathbf{2n}}\succeq G^{\scalebox{1.5}{$\boxslash$}}_{\mathbf{1},\mathbf{2n}}$, one can see that 
\begin{align}
    \frac{1}{2}\lambda_{1}^{(2n,2n,L,2)}\ \succeq \ \lambda_{1}^{(n,n-\frac{1}{2},L,4)}.\label{2,4 stoch}
\end{align}

This gives a stochastic comparison between the largest eigenvalues in Laguerre ensemble for $\beta=2$ and $\beta=4$.
\begin{remark}
    One can see from second part of Theorem \ref{Stoch dom} that \eqref{Laguerre} generalizes \eqref{1,2 stoch} and \eqref{2,4 stoch}. It would be interesting to know if such stochastic domination results in last passage for geometric weights could be obtained from Meixner ensemble (see \cite{KJ}).
\end{remark}

\subsection{Higher order analogues of Tracy Widom distributions:} The stochastic dominance result of \cite{VP22} was obtained using the description of $TW_{\beta}$ as the eigenvalue distribution of SAO.  It is a natural question whether such stochastic domination results hold for higher order analogues of Tracy Widom distributions. Using the Painlev\'e hierarchy, higher order analogues of $TW$ have been defined for $\beta=2$ in \cite{CIK10}. For general $\beta$, similar to SAO, the $k$-th higher order analog of $TW_{\beta}$ is defined as the negative of the smallest eigenvalue of the operator
\begin{align*}
    S_{\beta,k}:=-\frac{d^2}{dx^2}+x^{\frac{1}{2k+1}}+\frac{2}{\sqrt{\beta}}x^{-\frac{k}{2k+1}}b_x'
\end{align*}
on the half-line with Dirichlet conditions at the origin. Such operators are defined rigorously in the work of Krishnapur, Rider and Vir\'ag \cite{KRV13}. The authors in \cite{KRV13} define $TW_{\beta,k}$ as
\begin{align}\label{eq: TW_k}
TW_{\beta,k}:=\sup\limits_{\substack{f\in L^*,\  \lVert f\rVert_2=1}}\frac{2}{\sqrt{\beta}}\int\limits_{0}^{\infty}f^2(x)x^{-\frac{k}{2k+1}}dW_x-\int\limits_{0}^{\infty}\left[(f')^2(x)+x^{\frac{1}{2k+1}}f^2(x)\right]dx,
\end{align}
where $L^*$ is the Hilbert space of continuous $f$ satisfying $f(0)=0$ and $\int\limits_{0}^{\infty}\left[(f')^2(x)+x^{\frac{1}{2k+1}}f^2(x)\right]dx<\infty$. Note that for $k=0$, we get back SAO and $TW_2$. Using this definition and arguing similarly to \cite{VP22}, the following can be shown. 
\begin{align*}
    \mbox{If }\beta_2>\beta_1 \mbox{ and }{\frac{1}{4k+3}}\leq s\leq{\frac{4k+2}{4k+3}},
\mbox{ then } \beta_1^sTW_{\beta_1,k}\succeq\beta_2^s TW_{\beta_2,k}.
\end{align*}

In order to show that the above bound of $s$ is necessary for this stochastic domination result to hold, we need tail bounds for $TW_{\beta,k}$. We shall prove that the definition in \eqref{eq: TW_k} contradicts the tail bounds of $TW_{2,k}$. Hence \eqref{eq: TW_k} cannot be the correct definition of $TW_{\beta,k}$. This allows us to disprove Conjecture $13.1$ of \cite{KRV13} for a special case. To state the conjecture we need the following. 

Consider probability measures with density of the form 
\begin{align}\label{eq: general V}
    \frac{1}{Z_{n,\beta}}\exp{\l(-\beta n\sum_{k=1}^nV(\lambda_k)) \r)}\prod_{j<k}|\lambda_j-\lambda_k|^\beta
\end{align}
on $\R^n$, where $V$ is real analytic on $\R$ with sufficient growth conditions. For $\beta=2$, the limiting mean eigenvalue distribution $\mu_V$ for \eqref{eq: general V} is given as minimizer of a variational problem and has density of the form (see introductory part of \cite{CIK10} or Section $13$ of \cite{KRV13})
\begin{align*}
    \psi_V(x)=\frac{d\mu_V(x)}{dx}=\sqrt{(Q_V(x))_+}
\end{align*}
for a real analytic function $Q_V(x)$. Usually $Q_V(s)$ has simple zeros at the end points of support of $\psi_V(x)$, which gives that $\psi_V(x)$ vanishes as a square root. Such $V$ are called as regular potentials. In general $Q_V$ has a zero of order $4k+1$, with $k=0,1,2,\dots$ at the end point of support. For $k\neq 0$, these $V$ are called irregular potentials.

For each probability measure $\pi$ supported on $n$ points there exists $T$, a unique Jacobi matrix (tridiagonal with positive off-diagonals), such that the spectral measure of $T$ at $e_1$ is $\pi.$ Let $T_n$ be the random tridiagonal matrix corresponding to the random measure obtained such that the eigenvalues are sampled using \eqref{eq: general V} and weights independently sampled from Dirichlet$(\beta/2,\dots,\beta/2)$. It is conjectured in \cite{KRV13} that if $V$ is irregular and $b$ is the right endpoint in support of $\psi_V$, then there exist constant $a$ such that smallest eigenvalue of $an^{2/4k+3}(bI-T_n)$ converges weakly to smallest eigenvalue of $S_{\beta,k}$. In the second main result of the article we disprove a special case of this conjecture.

\begin{theorem}\label{Disprove conjecture result}
    Conjecture $13.1$ of \cite{KRV13} does not hold for $\beta=2$ and $k=1$.
\end{theorem}

In order to prove Theorem \ref{Disprove conjecture result}, we follow the right tail lower bound proof of $TW_{\beta}$ obtained in \cite{RRV11} and modify it for $S_{\beta,k}$. We obtain right tail lower bound for $TW_{\beta,2}$ from \cite{CV07} and show that these bounds contradict each other. We now return to Hermite and Laguerre ensembles.
\subsection{Moments comparison}
It can be seen from the proof method of Theorem \ref{Stoch dom} that such arguments work only for comparing $\lambda_{1}^{(m,H,\beta_1)}$ and $\lambda_{1}^{(n,H,\beta_2)}$ where $\frac{\beta_2}{\beta_1}=\frac{m}{n}$ and does not work for comparing $\lambda_{1}^{(n,H,\beta_1)}$ and $\lambda_{1}^{(n,H,\beta_2)}$. Also the proof method of Theorem \ref{Stoch dom} does not work to stochastically compare other eigenvalues except largest eigenvalue. Due to these reasons, we study $\E_{n,H,\beta}\l[ \sum\limits_{i=1}^n \lambda_i^p \r]$ ($p$-th moments of Hermite ensemble) and $\E_{n,\kappa,L,\beta}\l[ \sum\limits_{i=1}^n \lambda_i^p \r]$ ($p$-th moments of Laguerre ensemble) where $\lambda_i$ are from \eqref{Hermite density} and \eqref{Laguerre density} respectively. Note that these expectations are equal to $\E\l[\mbox{Tr}(T_{n,\beta}^{p})\r]$ and $\E\l[\mbox{Tr}((B^t_{n,\kappa, \beta}B_{n,\kappa, \beta})^{p})\r]$ respectively. For special cases of $\beta=1,2,4$ the moments of Hermite ensemble admit a combinatorial interpretation and are related to certain maps on surfaces of definite genus \cite{HZ86, MW03}. Also the moments of Laguerre $\beta$ ensemble are related to alternating Motzkin paths (see Chapters 5-6 of \cite{Dumitriu03}). As we could not do stochastic comparison of all the eigenvalues, we use these moments as a measure of global eigenvalues ordering and compare these moments for different $\beta$. In this direction, we have the following theorem.

\begin{theorem}\label{thm:Expectation comparison}
    Let $p\in \N$ and $\beta_2>\beta_1$. Then
    \begin{align*}
        \E\l[\mbox{Tr}\l(T_{n,\beta_1}^{p}\r)\r]\geq \E\l[\mbox{Tr}\l(T_{n,\beta_2}^{p}\r)\r] \quad \text{ and }\quad \E\l[\mbox{Tr}\l(\l(B^t_{n,\kappa, \beta_1}B_{n,\kappa, \beta_1}\r)^{p}\r)\r]\geq \E\l[\mbox{Tr}\l(\l(B^t_{n,\kappa, \beta_2}B_{n,\kappa, \beta_2}\r)^{p}\r)\r]
    \end{align*}
    and if $\frac{\beta_2}{\beta_1}=\frac{m}{n}$, then
    \begin{align*}
        \E\l[\mbox{Tr}\l(T_{m,\beta_1}^{p}\r)\r]\geq \l(\frac{\beta_2}{\beta_1}\r)^{p/2} \E\l[\mbox{Tr}\l(T_{n,\beta_2}^{p}\r)\r] \text{ and } \E\l[\mbox{Tr}\l(\l(B^t_{m,m\kappa/n, \beta_1}B_{m,m\kappa/n, \beta_1}\r)^{p}\r)\r]\geq \l(\frac{\beta_2}{\beta_1}\r)^{p}\E\l[\mbox{Tr}\l(\l(B^t_{n,\kappa, \beta_2}B_{n,\kappa, \beta_2}\r)^{p}\r)\r].
    \end{align*}
\end{theorem}

For the special case of $\beta=2$, we denote $a_p(n):=\E\l[\mbox{Tr}\l(T_{n,2}^{2p}\r)\r]$ (odd moments are $0$ due to symmetry) and $c_p(m,n):=\E\l[\mbox{Tr}\l(\l(B^t_{n,m, \beta}B_{n,m,\beta}\r)^{p}\r)\r]$. These moments satisfy the following three term recursions, due to Harer-Zagier \cite{HZ86} and Haagerup-Thorbj\o{}rnsen \cite{HT03}.
\begin{align}
    (p+2)a_{p+1}(n)\ & =\ 2n(2p+1)a_{p}(n)\ +\ p(2p+1)(2p-1)a_{p-1}(n).\label{eq:Harer-Zagier}\\
    (p+2)c_{p+1}(n)\ &=\ (m+n)(2p+1)c_{p}(m,n)\ +\ (p-1)(p^2-(m-n)^2)c_{p-1}(m,n).\label{eq:Haagerup -Thorbjornsen}
\end{align}

Similarly there also exist five term recursions for $\beta=1$ case \cite{Ledoux09, CMOS19}.
The first few Hermite moments are given as,
\begin{align}
    a_0(n)&=n,\nonumber\\
    a_1(n)&=n^2,\nonumber\\
    a_2(n)&=2n^3+n,\nonumber\\
    a_3(n)&=5n^4+10n^2,\nonumber\\
    a_4(n)&=14n^5+70n^3+21n.\label{eq:Hermite moments}
\end{align}

The first few Laguerre moments are given as,
\begin{align}
    c_0(m,n)&=n,\nonumber\\
    c_1(m,n)&=mn,\nonumber\\
    c_2(m,n)&=m^2n+mn^2,\nonumber\\
    c_3(m,n)&=m^3n+3m^2n^2+mn^3+mn,\nonumber\\
c_4(m,n)&=m^4n+6m^3n^2+6m^2n^3+mn^4+5m^2n+5mn^2.\label{eq:Laguerre moments}
\end{align}

Note that as the celebrated Wigner's semicircle law and Marchenko-Pastur law are the limiting distributions of Hermite and Laguerre ensembles, it is a natural question to compare $\E_{n,H,\beta}\l[ \sum\limits_{i=1}^n \lambda_i^{2p} \r]$ and $\E_{n,n,L,\beta}\l[ \sum\limits_{i=1}^n \lambda_i^{p} \r]$. Note that for $\beta=2$, using \eqref{eq:Harer-Zagier} and \eqref{eq:Haagerup -Thorbjornsen}, it is immediate that $\E_{n,H,2}\l[ \sum\limits_{i=1}^n \lambda_i^{2p} \r]\geq \E_{n,n,L,2}\l[ \sum\limits_{i=1}^n \lambda_i^{p} \r]$. Similarly comparing first few moments for $\beta=1$ using the five term recursions mentioned above, we make the following conjecture.
\begin{conjecture}
For any $\beta>0$ and $p\in \N$, we have $\E_{n,H,\beta}\l[ \sum\limits_{i=1}^n \lambda_i^{2p} \r]\geq \E_{n,n,L,\beta}\l[ \sum\limits_{i=1}^n \lambda_i^{p} \r]$.
\end{conjecture}

\section{Proof of Theorem \ref{Stoch dom}}\label{Sec 2}
\begin{proof}[Proof of Theorem \ref{Stoch dom}]
As $s=2/3$ and $m\beta_1=n\beta_2$, to prove \eqref{Hermite} it is enough to prove $\lambda_1\l(\sqrt{\beta_1}T_{m,\beta_1}\r)  \succeq \  \lambda_1\l(\sqrt{\beta_2}T_{n,\beta_2}\r)$.

As the largest eigenvalue of $\sqrt{\beta_1}T_{m,\beta_1}$ is at least as large as that of its top principal submatrices, we would be done if we show that the largest eigenvalue of top $n\times n$ submatrix of $\sqrt{\beta_1}T_{m,\beta_1}$ stochastically dominates $\lambda_1\l(\sqrt{\beta_2}T_{n,\beta_2}\r)$.
\begin{align*}
\centering
 A_{n}=\l[\sqrt{\beta_1}T_{m,\beta_1}\r]_{n\times n}:=&\begin{bmatrix}
X_1 & \nu_{m\beta_1-\beta_1} & 0 & 0 &\cdots 0\\
\nu_{m\beta_1-\beta_1} & X_2 & \nu_{m\beta_1-2\beta_1} & 0 &\cdots 0\\
0 & \nu_{m\beta_1-2\beta_1} & X_3 & \nu_{m\beta_1-3\beta_1} & \cdots 0\\
\vdots & \vdots & \vdots &\vdots & \vdots\\
\vdots & \vdots & \vdots &\nu_{m\beta_1-(n-1)}\beta_1 & X_n
\end{bmatrix}_{n\times n}\\
\centering B_n=\sqrt{\beta_2}T_{n,\beta_2}:=&\begin{bmatrix}
X_1 & \nu'_{n\beta_2-\beta_2} & 0 & 0 &\cdots 0\\
\nu'_{n\beta_2-\beta_2} & X_2 & \nu'_{n\beta_2-2\beta_2} & 0 &\cdots 0\\
0 & \nu'_{n\beta_2-2\beta_2} & X_3 & \nu'_{n\beta_2-3\beta_2} & \cdots 0\\
\vdots & \vdots & \vdots &\vdots & \vdots\\
\vdots & \vdots & \vdots &\nu'_{\beta_2} & X_n
\end{bmatrix}_{n\times n}
\end{align*}
Consider the above matrices where $A_n$ is the top $n\times n$ submatrix of $\sqrt{\beta_1}T_{m,\beta_1}$. In the sub-diagonal and super-diagonal entries the chi-distributed random variables are written using symbols $\nu,\nu'$ instead of $\chi$. Note that same normal random variables are used in the diagonals in both matrices. The matrices $A_n$ and $B_n$ are coupled such that $A_n(i,i+1)\geq B_n(i,i+1), \forall i\in [n]$, almost surely. As $m\beta_1=n\beta_2$, along with $\beta_1<\beta_2$ and $\chi_{a} \succeq \ \chi_{b}$ if $a>b$, such a coupling exists.

As both the matrices are tridiagonal and off-diagonal entries are positive, the eigenvector corresponding to the largest eigenvalue have all co-ordinates non-negative. Since $A_n-B_n$ has non-negative entries (almost surely), if $x\in \mathbb{R}^n$ has all entries non-negative then $x^tA_nx\geq x^tB_nx$. In particular, if $x\in\mathbb{R}^n$ is the eigenvector corresponding to the largest eigenvalue of $B_n$, we have $\lambda_1(A_n){\geq}\lambda_1(B_n)$ (almost surely).

We now prove \eqref{Laguerre}. We have that
 \begin{align*}
\beta_1^{2/3}\left(\sqrt{\frac{m^2\kappa}{n}}\right)^{1/3}\left(\sqrt{\frac{m\kappa}{n}}+\sqrt{m}\right)^{2/3}&=\beta_2^{2/3}\left(\sqrt{\kappa n}\right)^{1/3}\left(\sqrt{\kappa}+\sqrt{n}\right)^{2/3}\\
\beta_1\ \left(\sqrt{\frac{m\kappa}{n}}+\sqrt{m}\right)^{2}&=\beta_2\ \left(\sqrt{\kappa}+\sqrt{n}\right)^{2}.
\end{align*} 

The above equations show that to prove \eqref{Laguerre}, it is enough to show 
\begin{align*}
\beta_1\lambda_1(B^t_{m,m\kappa/n,\beta_1}B_{m,m\kappa/n,\beta_1})  \succeq \  \beta_2\lambda_1(B^t_{n,\kappa,\beta_2}B_{n,\kappa,\beta_2}).
\end{align*}
As $\beta_1(\frac{m\kappa}{n}-i+1)>\beta_2({\kappa}-i+1)$ and $\beta_1(m-i)>\beta_2(n-i)$ we can couple the entries of $B_{m,m\kappa/n,\beta_1}$ and $B_{n,\kappa,\beta_2}$ such that each entry of $B_{n,\kappa,\beta_2}$ is smaller than the corresponding entry in the top $n\times n$ sub-matrix of $B_{m,m\kappa/n,\beta_1}$. For the rest of the proof following the same argument as in the proof of \eqref{Hermite} gives us \eqref{Laguerre}.

\end{proof}

\begin{proof}[Proof of Corollary \ref{Pedreira result}]
    Letting $n\rightarrow\infty$ and using the property that $TW_{\beta_n}\rightarrow TW_{\beta}$ weakly if $\beta_n\rightarrow\beta$, Theorem \ref{Stoch dom} gives that if $\beta_2>\beta_1$ then
    \begin{align*}
     \beta_1^{2/3}\ TW_{\beta_1}\ \succeq\ \ \beta_2^{2/3}\ TW_{\beta_2}.
    \end{align*}
    This gives Theorem $1$ of \cite{VP22} for $\alpha=\left(\frac{\beta'}{\beta}\right)^{2/3}$. Note that it is enough to check Theorem $1$ of \cite{VP22} for $\alpha=\left(\frac{\beta'}{\beta}\right)^{1/3}$ and $\left(\frac{\beta'}{\beta}\right)^{2/3}$.
\end{proof}

\begin{proof}[Proof of Theorem \ref{thm: TW neg means}]
 We make use of the following result from \cite{AD13} that deformed $TW_{\beta}$ converges to Gumbel distribution as $\beta\rightarrow 0$. 

\begin{theorem}[Theorem $4.2$ of \cite{AD13}]\label{thm: Gumble conv}
When $\beta\rightarrow 0$, the following convergence in law holds
\begin{align}
    X_{\beta}:=2.3^{1/3}\l(\ln\frac{1}{\beta}\r)^{1/3}\left[\l(\frac{\beta}{4}\r)^{2/3}TW_{\beta}-\l(\frac{3}{8}\r)^{2/3}\l(\ln\frac{1}{\pi\beta}\r)^{2/3}\right]\implies e^{-x}\exp(-e^{-x})dx
\end{align}
\end{theorem}

Note that the mean $\mu$ of the standard Gumbel distribution is the Euler-Mascheroni constant ($\gamma$), that is $\mu=\gamma\approx 0.5772$. By the above result one would guess that the moments also converge and as a result for $\beta$ small enough, $\mathbb{E}[TW_{\beta}]>0$. But the weak convergence does not necessarily imply convergence of moments. Also to get the convergence of $k$-th moments, it is rather difficult to bound the $(k+\epsilon)$ moments of $X_{\beta}$. Here we use the log-concavity of $X_{\beta}$ to conclude the convergence of all integer moments. By  \cite[Corollary $6$]{MM14}, we have that moments of log-concave measures converge to the corresponding moments of the limiting distribution, which is also log-concave. 

By \cite[Corollary $4$]{BKM24}, we have that for all $\beta>0$, the  $TW_{\beta}$ distributions are log-concave. As log-concavity is preserved under centering, scaling and weak limits, we have that the distributions $X_{\beta}$ and Gumbel distribution are log-concave. As a result we have that $\mathbb{E}[X_{\beta}]\rightarrow \gamma$ as $\beta\rightarrow 0$. It follows from that for some $0<a<1$, we have $\mathbb{E}[TW_{a}]>0.$ By Corollary \ref{Pedreira result}, for all $\beta<a$, we have $\mathbb{E}[X_{\beta}]>0$.

In random interface growth models literature \cite{QR19} stochastic domination of Baik-Rains distribution ($BR$) over $TW_{1}$ and $TW_{2}$ is known. It is also known that $\mathbb{E}[BR]=0$ \cite[Proposition 2.1]{BR00}. This implies $\mathbb{E}[TW_1]<0$. This was shown in \cite[Lemma A.4]{BGHH22}. By Corollary \ref{Pedreira result}, we also have that $\mathbb{E}[TW_{\beta}]<0$ for all $\beta\geq 1$. 

\end{proof}

\section{Proof of Theorem \ref{Disprove conjecture result}}
\begin{proof}[Proof of Theorem \ref{Disprove conjecture result}]

 We describe the following alternate definition of higher order analogue of Tracy Widom. Building on the Riemann-Hilbert analysis of the limiting kernel in \cite{CV07}, Claeys, Its and Krasovsky \cite{CIK10} use Fredholm determinant theory to describe higher order analogues of $TW_{2}$. For $\beta=2$, we obtain a determinantal point process with kernel $K_{n,V}$ when points are distributed with density \eqref{eq: general V} \cite{Deift00}.
 
 It is shown in \cite{CV07} that if $V$ is a polynomial such that $\psi_V\sim (b-x)^{5/2}$ at the right end of support of $\psi_V$, then there exists a limiting kernel $K^{(1)}(u,v)$ and a constant $c=c(V)$ such that
 \begin{align*}
\lim_{n\rightarrow\infty}\frac{1}{cn^{2/7}}K_{n,V}\l(b+\frac{u}{cn^{2/7}},b+\frac{v}{cn^{2/7}}\r)=K^{(1)}(u,v)
 \end{align*}
uniformly for $u$ and $v$ in compact subsets of the real line. Since both sides of this equation tend to $0$ rapidly as $u$ or $v$ tends to $+\infty$, this result can be extended to an uniform statement for $u,v>-L$ for a large $L$. This implies (see Equation $1.19$ of \cite{CS12})
\begin{align*}
\lim_{n\rightarrow\infty}\P\l(cn^{2/7}\l(\lambda_1-b\r)\leq s\r)= \det \l(I-K_s^{(1)}\r)
\end{align*}
where $K_s^{(1)}$ is the trace class operator with kernel $K^{(1)}$ on $(s,\infty).$  

Although no proofs have been given for $k>1$, it has been conjectured that (see Equation $1.13$ and $1.15$ of \cite{CIK10}) there exist limiting kernels $K^{(k)}(t_0,\dots,t_{2k-1})$ such that,
\begin{align*}
\lim_{n\rightarrow\infty}\frac{1}{cn^{2/4k+3}}K_{n,V}\l(b+\frac{u}{cn^{2/4k+3}},b+\frac{v}{cn^{2/4k+3}}\r)&=K^{(k)}(u,v; t_0,\dots, t_{2k-1}),\\
\lim_{n\rightarrow\infty}\P\l(cn^{2/4k+3}\l(\lambda_1-b\r)\leq s\r)&= \det \l(I-K_s^{(k)}(t_0,\dots,t_{2k-1})\r).
 \end{align*}
Here $K_s^{(k)}$ are certain kernels which are generalizations of Airy kernel, as given in \cite{CIK10}. 

 For a few regular potential $V$, it is shown that there exist constants $b,c$ such that 
\begin{align*}
\lim_{n\rightarrow\infty}\P\l(cn^{2/3}\l(\lambda_{1}-b\r)\leq s\r)=\P(TW_2\leq s)=\det \l(I-K_s^{(0)}\r),
\end{align*}
where $K_s^{(0)}$ is the Airy kernel trace class operator acting on $L^2(s,\infty)$. Higher order analogues of $TW_2$ denoted as $TW_{2,k} $ (for $k=0$ we get $TW_2$ distribution) are defined as (see \cite{CIK10})
\begin{align*}
\lim_{n\rightarrow\infty}\P\l(cn^{2/(4k+3)}\l(\lambda_{1}-b\r)\leq s\r)=\P(TW_{2,k}\leq s)=\det \l(I-K_s^{(k)}\r).
\end{align*}

It is also known that (Equation $1.24$ of \cite{CIK10})
\begin{align*}
    \ln \det \l(I-K_s^{(k)}\r)=O\l(\exp\l(-cs^{\frac{4k+3}{2}}\r)\r),
\end{align*}
for some constant $c>0$. This implies 
\begin{align}\label{eq: right tail upper bound}
\lim_{n\rightarrow\infty}\P\l(cn^{2/7}\l(\lambda_1-b\r)\geq s\r)= O\l(\exp\l(-cs^{7/2}\r)\r).
\end{align}

We now obtain lower bound for the right tail of $-\Lambda_0$ for the operator $S_{\beta,k}$. For any $a>0$ and $g$ satisfying $g(0)=0$ and $\int\limits_{0}^{\infty}\left[(g')^2(x)+x^{\frac{1}{2k+1}}g^2(x)\right]dx<\infty$., we have that \begin{align}
  \P\l(-\Lambda_0>a\r)&=\P\l(\Lambda_0<-a\r),\nonumber\\
  &\geq \P\l(\langle g,S_{\beta,k}g\rangle\leq -a\langle g,g\rangle\r),\nonumber\\ \label{eq: TW_{2,k} tail bound}
  &=\P\l(\frac{2}{\sqrt{\beta}}\sqrt{\int_{0}^{\infty}x^{-\frac{2k}{2k+1}}g^4(x)dx}\ Z\leq -a\int_0^{\infty}g^2(x)dx-\int_0^{\infty}\l(g'(x)\r)^2dx-\int_0^{\infty}x^{\frac{1}{2k+1}}g^2(x)dx\r).
\end{align}
Here $Z$ is a standard normal random variable. Here we used the fact that for any deterministic $f$, we have $\int\limits_0^{\infty}fdW_x\sim \mbox{N}(0,\lVert f\rVert_2^2)$. 

Fix any function $f$ which is smooth and compactly supported on an interval away from $0$. We write $f_a(x)=f(\sqrt{a}x)$. We define
\begin{align}\label{eq: integral equalities}
   \frac{A}{\sqrt{a}}:= \int_0^{\infty}f_a^2(x)dx&=\frac{1}{\sqrt{a}}\int_0^{\infty}f^2(x)dx,\\
    B\sqrt{a}:=\int_0^{\infty}(f'_a(x))^2dx&={\sqrt{a}}\int_0^{\infty}(f'(x))^2dx,\\
    \frac{C}{{a}^{\frac{k+1}{2k+1}}}:=\int_0^{\infty}x^{\frac{1}{2k+1}}f_a^2(x)dx&=\frac{1}{{a}^{\frac{k+1}{2k+1}}}\int_0^{\infty}x^{\frac{1}{2k+1}}f^2(x)dx,\\
  \frac{D}{{a}^{\frac{1}{4k+2}}}:=\int_0^{\infty}x^{-\frac{2k}{2k+1}}f_a^4(x)dx&=\frac{1}{{a}^{\frac{1}{4k+2}}}\int_0^{\infty}x^{-\frac{2k}{2k+1}}f^4(x)dx .
\end{align}

Using $g=f_a(x)$ for $a$ in \eqref{eq: TW_{2,k} tail bound} and using above equations we get
\begin{align*}
    \P\l(\frac{2}{\sqrt{\beta}}\frac{\sqrt{D}}{\sqrt{a}^{\frac{1}{4k+2}}}Z\geq \sqrt{a}\l(A+B+\frac{C}{a^{\frac{4k+3}{4k+2}}}\r)\r)\geq \exp\l(-\l(\frac{A+B+o(1)C}{D}\r)^2\frac{\beta}{8}a^{\frac{4k+3}{4k+2}}\r).
\end{align*}

Hence there exists a constant $c_{\beta}>0$ such that for all large enough $a$
\begin{align}\label{eq: right tail lower bound for eigenvalue}
    \P\l(-\Lambda_0>a\r)\geq \exp\l(-c_{\beta}a^{\frac{4k+3}{4k+2}}\r).
\end{align}

Note that if we take $-\Lambda_0$ to be the limiting distribution of $cn^{2/7}\l(\lambda_1-b\r)$ for $\beta=2$ and $k=1$, then \eqref{eq: right tail upper bound} and \eqref{eq: right tail lower bound for eigenvalue} contradict each other. Hence $cn^{2/7}\l(\lambda_1-b\r)$ does not converge to $-\Lambda_0$ and this completes the proof of Theorem \ref{Disprove conjecture result}.
\end{proof}

\begin{remark}
    It also follows from a heuristic argument (see Section $5$ of \cite{MS14}) that for $k>1$ 
    \begin{align*}
        \lim_{n\rightarrow\infty}\P\l(cn^{2/(4k+3)}\l(\lambda_{1}-b\r)\geq a\r)= O\l(\exp\l(-c_ka^{\frac{4k+3}{2}}\r)\r).
    \end{align*}
    This along with \eqref{eq: right tail lower bound for eigenvalue} suggests that Conjecture $13.1$ of \cite{KRV13} should not be true in general. 
\end{remark}

\section{Proof of Theorem \ref{thm:Expectation comparison}}
\begin{proof}[Proof of Theorem \ref{thm:Expectation comparison}]
For odd $p$ we have $\E\l[\mbox{Tr}\l(T_{n,\beta}^{p}\r)\r]=0$. For even $p$, we can write 
\begin{align}\label{eq: moment as sum}
\E\l[\mbox{Tr}\l(T_{n,\beta}^{p}\r)\r]=\frac{1}{\beta^{p/2}}\sum_{i}\sum_{P_i}\prod_{j=1}^{n}\E\l[X_j^{f(P_i,j)}\r]\prod_{k=1}^{n-1}\E\l[Y_k^{g(P_i,k)}\r].
\end{align}
Here
 $P_i$ is a closed walk of length $p$ starting and ending at vertex $i$ on the graph $G$ as shown in the Figure \ref{Fig: 1} and traversing each loop even number of times.

 \begin{figure}[h!]
 \centering
\begin{tikzpicture}
  
\Vertex[label=$1$,x=0,y=0]{A}
\Vertex[label=$2$,x=2,y=0]{B}
  \Vertex[label=$n-1$,x=6,y=0]{C}
  \Vertex[label=$n$,x=8,y=0]{D}

  \Edge(A)(B)
  \draw[dashed, ultra thick] (B) -- (C);
  \Edge(C)(D)

  \draw (A) to [out=45, in=135, looseness=10] (A);
  \draw (B) to [out=45, in=135, looseness=10] (B);
  \draw (C) to [out=45, in=135, looseness=10] (C);
  \draw (D) to [out=45, in=135, looseness=10] (D);
\end{tikzpicture}
\caption{Graph $G$}
\label{Fig: 1}
\end{figure}
The function $f(P_i,j)$ counts the number of times the loop at vertex $j$ is covered by the path $P_i$ and function $g(P_i,k)$ counts the number of times the path $P_i$ travels the edge connecting $k$ and $k+1$. Note that $\frac{\E\l[Y_k^{2m}\r]}{\beta_1^{m}}$ in $T_{n,\beta_1}$ is greater than $\frac{\E\l[Y_k^{2m}\r]}{\beta_2^{m}}$ in $T_{n,\beta_2}$ for all $m\in\N$ and $k\in[n-1]$. We can compare $\E\l[\mbox{Tr}\l(T_{n,\beta_1}^{p}\r)\r]$ and $ \E\l[\mbox{Tr}\l(T_{n,\beta_2}^{p}\r)\r]$ termwise using \eqref{eq: moment as sum}. It follows that 
\begin{align*}
    \E\l[\mbox{Tr}\l(T_{n,\beta_1}^{p}\r)\r]\geq \E\l[\mbox{Tr}\l(T_{n,\beta_2}^{p}\r)\r].
\end{align*}
Using the fact that $B^t_{n,\kappa, \beta}B_{n,\kappa, \beta}$ is also a tridiagonal matrix and following essentially the same argument as above it follows that
\begin{align*}
    \E\l[\mbox{Tr}\l(\l(B^t_{n,\kappa, \beta_1}B_{n,\kappa, \beta_1}\r)^{p}\r)\r]\geq \E\l[\mbox{Tr}\l(\l(B^t_{n,\kappa, \beta_2}B_{n,\kappa, \beta_2}\r)^{p}\r)\r].
\end{align*}

Now let $\frac{\beta_2}{\beta_1}=\frac{m}{n}$ and $p$ be even. As all the moments of all entries of $T_{m,\beta}$ are non-negative, it follows that if $[T_{m,\beta_1}]_n$ is the top $n\times n$ sub-matrix of $T_{m,\beta_1}$ then $\E\l[\mbox{Tr}\l(T_{m,\beta_1}^{p}\r)\r]\geq \E\l[\mbox{Tr}\l([T_{m,\beta_1}]_n^{p}\r)\r]$. As shown in the proof of Theorem \ref{Stoch dom}, each entry of $[T_{m,\beta_1}]_n$ stochastically dominates corresponding entry of $\l(\frac{\beta_2}{\beta_1}\r)^{1/2}T_{n,\beta_2}$. As a result, using \eqref{eq: moment as sum} and comparing termwise gives us,
\begin{align*}
    \E\l[\mbox{Tr}\l(T_{m,\beta_1}^{p}\r)\r]\geq \l(\frac{\beta_2}{\beta_1}\r)^{p/2} \E\l[\mbox{Tr}\l(T_{n,\beta_2}^{p}\r)\r].
\end{align*}  
Arguing similarly for $B^t_{n,\kappa, \beta}B_{n,\kappa, \beta}$ it follows that
\begin{align*}
     \E\l[\mbox{Tr}\l(\l(B^t_{m,m\kappa/n, \beta_1}B_{m,m\kappa/n, \beta_1}\r)^{p}\r)\r]\geq \l(\frac{\beta_2}{\beta_1}\r)^{p}\E\l[\mbox{Tr}\l(\l(B^t_{n,\kappa, \beta_2}B_{n,\kappa, \beta_2}\r)^{p}\r)\r].
\end{align*}  
\end{proof}

\textbf{Acknowledgement:} The author would like to thank Riddhipratim Basu and Manjunath Krishnapur for many helpful discussions and suggestions.

 \bibliography{bibliography}
 \bibliographystyle{plain}

\end{document}